\documentclass[10pt]{article}
\usepackage{latexsym,amsmath,graphicx,amssymb,mathrsfs}
\usepackage{CJK}

\newtheorem{theo}{Theorem}[section]

\newtheorem{lemm}[theo]{Lemma}
\newtheorem{prop}[theo]{Proposition}

\newtheorem{conj}[theo]{Conjecture}

\numberwithin{equation}{section}

\newdimen\Squaresize \Squaresize=11pt
\newdimen\Thickness \Thickness=0.7pt
\def\Square#1{\hbox{\vrule width \Thickness
   \vbox to \Squaresize{\hrule height \Thickness\vss
    \hbox to \Squaresize{\hss#1\hss}
   \vss\hrule height\Thickness}
\unskip\vrule width \Thickness} \kern-\Thickness}

\def\Vsquare#1{\vbox{\Square{$#1$}}\kern-\Thickness}

\def\moins{\raise 1pt\hbox{{$\scriptstyle -$}}}
\begin{document}
\begin{CJK*}{GBK}{song}
\title{Proof of a Conjecture of Z.-W. Sun on  Trigonometric Series\thanks{This work is supported by NSFC (No. 11171283).}}
\author{Brian Y. Sun\thanks{Corresponding author.
{\small Email: brianys1984@126.com.}} ~and J. X. Meng\\
{\small College of Mathematics and System Science,}\\
{\small Xinjiang University, Urumqi 830046, P.R.China}
}
\date{}
\maketitle {\flushleft\bf Abstract.}~Recently, Z. W. Sun introduced a sequence $(S_n)_{n\geq
0}$, where
$S_n=\frac{\binom{6n}{3n}
\binom{3n}{n}}{2(2n+1)\binom{2n}{n}}$,
 and found one congruence and two convergent series
 on $S_n$ by {\tt{Mathematica}}. Furthermore, he proposed
  some related conjectures. In this paper, we first
  give analytic proofs of his two convergent series and then
   confirm one of his conjectures by invoking series expansions of $\sin(t\arcsin(x))$ and
$\cos(t\arcsin(x)).$

{\flushleft{\bf Key words:} Divisibility; Congruences; Trigonometric series;
Convergent series}
{\flushleft{\bf AMS Classification 2010:} Primary 11B65; Secondary 05A10, 11A07.}

\section{Introduction}
 \label{sec:introduction} Throughout the paper, we let $\mathbb{R}$,
$\mathbb{N}=\{0,1,2,3,\ldots\}$ and $\mathbb{N}^+=\mathbb{N}\backslash \{0\}$ denote the set of
real numbers, natural numbers and positive integer numbers,
respectively.

Recently, Z. W. Sun \cite{sun} considered the divisibilities
 of products and sums concerned binomial
 coefficients and central binomial coefficients. He also studied the divisibility of $\binom{6n}{3n}\binom{3n}{n}$ and obtained the following result,
\begin{equation}\label{sun-divisibility}
2(2n+1)\binom{2n}{n}\Big |\binom{6n}{3n}\binom{3n}{n}\,\,\text{for all}\,\, n\in \mathbb{N}^+,
\end{equation}
where $\binom{m}{n}=\frac{m!}{n!(m-n)!}$ is a binomial coefficient. We call
$\binom{2n}{n}$ the central binomial coefficient (cf. \cite[A000984]{Sloane}). The binomial coefficients ${\binom{6n}{3n}}$ is the sequence \cite[A066802]{Sloane} and ${\binom{3n}{n}}$ is the sequence \cite[A005809]{Sloane}.

According to \eqref{sun-divisibility}, Z. W. Sun \cite{sun}
introduced the following integer sequence \cite[A176898]{Sloane}.
\begin{equation*}
S_n=\frac{\binom{6n}{3n}\binom{3n}{n}}
{2(2n+1){\binom{2n}{n}}},
\,\, \text{for} \,\,n\in \mathbb{N}^+.
\end{equation*}
Here we list the values of $S_1,S_2,\ldots,S_8$ as follows:
\begin{center}
5, 231, 14568, 1062347, 84021990,\\
7012604550, 607892634420, 54200780036595.
\end{center}

Z. W. Sun \cite{sun} proved that  for any odd prime $p$,
$$S_p\equiv15-30p+60p^2\pmod {p^3}.$$
Additionally, Guo  \cite{guo} studied the sequence and proved that
$3S_n\equiv 0~\pmod {2n+3}$, positively answering a question of Z.
W. Sun \cite{sun}. By setting $S_0=\frac{1}{2}$ and employing
{\tt{Mathematica}}, Z. W. Sun also obtained
\begin{equation}\label{mathematica-find-1}
\sum_{n\geq 0}S_nx^n=\frac{\sin(\frac{2}{3}\arcsin(6\sqrt{3x}))}{8\sqrt{3x}}\left(0<x\leq\frac{1}{108}\right),
\end{equation}
and
\begin{equation}\label{mathematica-find-2}
\sum_{n\geq 0}\frac{S_n}{(2n+3)108^n}=\frac{27\sqrt{3}}{256}.
\end{equation}
Particularly, setting  $x=1/108$ in \eqref{mathematica-find-1}, we
derive
\begin{equation}\label{convergent-bound}
\sum_{n=0}^\infty\frac{S_n}{108^n}=\frac{3\sqrt{3}}{8}.
\end{equation}
Moreover, he proposed the following conjecture.
\begin{conj}\label{conjecture-1}
There exist positive integers $T_1,T_2,\ldots$ such that
\begin{equation}\label{eq-conjecture-1}
\sum_{k=0}^\infty S_kx^{2k+1}+\frac{1}{24}-\sum_{k=1}^{\infty}T_kx^{2k}
=\frac{\cos(\frac{2}{3}\arccos(6\sqrt{3}x))}{12}
\end{equation}for all real $x$ with $|x|\leq \frac{1}{6\sqrt{3}}$.
Also, $T_p\equiv-2\pmod p$ for any prime $p$.
\end{conj}
In this paper, we first deduce formulas \eqref{mathematica-find-1}
 and \eqref{mathematica-find-2} by utilizing a series
  expansion in \cite{stanley}. Besides, we confirm
     Conjecture \ref{conjecture-1} , i.e.,
\begin{theo}\label{theorem}
Conjecture \ref{conjecture-1} is true.
\end{theo}

\section{Trigonometric Series}
\label{sec:Trigonometric Series}
Before proving Theorem \ref{theorem}, we need some formulas and series expansions on
trigonometric functions. We first note that
\begin{equation}\label{arcsin+arccos}
\arcsin(x)+\arccos(x)=\frac{\pi}{2},
\end{equation}
where $x\in [-1,1]$. What's more, the basic fact
\begin{equation}\label{cos(a+b)}
\cos(\alpha-\beta)=\cos\alpha\cos\beta+\sin\alpha\sin\beta
\end{equation}
is needed.

Here we also use two trigonometric series in \cite[Ex. \ 44, p. \ 51]{stanley}, namely,
\begin{equation}\label{stanley-series-1}
\sin(t \arcsin(x))=\sum_{n\geq 0}(-1)^nt\left(\prod_{i=0}^{n-1}(t^2-(2i+1)^2)\right)\frac{x^{2n+1}}{(2n+1)!},
\end{equation}
and
\begin{equation}\label{stanley-series-2}
\cos(t \arcsin(x))=\sum_{n\geq 0}(-1)^n\left(\prod_{i=0}^{n-1}(t^2-(2i)^2)\right)\frac{x^{2n}}{(2n)!}.
\end{equation}
With the formula \eqref{stanley-series-1} in hand,
 it is not difficult to derive
the identities  \eqref{mathematica-find-1} and
\eqref{mathematica-find-2}.

To begin with, we consider the uniform convergence of series
 \eqref{mathematica-find-1}.
Due to Stirling \cite{PD}, we hold the following approximate
formula, which was called Stirling's formula,
\begin{equation}\label{stirling-formula}\Gamma(\alpha)\approx \left(\frac{\alpha-1}{e}\right)^{\alpha-1}\sqrt{2\pi(\alpha-1)}
,\,\,\text{as}\,\,\alpha\rightarrow\infty,
\end{equation}
where $\Gamma(\alpha)$ is the gamma function and is defined by
$$\Gamma(\alpha)=\int_0^{+\infty}x^{\alpha-1}e^{-x}dx\,\, \mbox{for~}
\alpha>0.$$ For more information on this formula, one can consult
\cite{PD}. By Stirling's formula \eqref{stirling-formula}, we have
\begin{equation}\label{convergent-radius}
\rho=\overline{\lim_{n\rightarrow\infty}}\sqrt[n]{S_n}=108,\,\,\text{as}
\,\,n\rightarrow \infty.
\end{equation}
Thus by Cauchy-Hadamard's theorem and \eqref{convergent-radius}, the
series  $\sum_{n=0}^\infty S_nx^n$ is uniformly convergent for
$0<x<\frac{1}{108}$. So, before we make operations on
$\sum_{n=0}^\infty S_nx^n$, we  designate  $x\in(0,1/108)$ in order
that all the following operations are well defined.

\begin{theo}\label{in-int}For all $a,b\in \mathbb{R}$ and $bx\in[-1,1]$, we have
$$\int x\sin(a\arcsin(bx))dx=\frac{\frac{\sin \left((a-2) \arcsin(b x)\right)}{a-2}-\frac{\sin \left((a+2) \arcsin(b x)\right)}{a+2}}{4 b^2}+C,$$
where $C$ is any constant.
\end{theo}

Now we are in a position to consider formulas \eqref{mathematica-find-1} and \eqref{mathematica-find-2}.
\begin{itemize}
\item[(i)] As to   \eqref{mathematica-find-1}, let $t=2/3$ and $x=6\sqrt{3x}$ in \eqref{stanley-series-1}, we obtain
    \begin{equation}\label{sn}
    \begin{split}
    \sin(\frac{2}{3} \arcsin(6\sqrt{3x}))&=\sum_{n\geq 0}\frac{2}{3}
    \left(\prod_{i=0}^{n-1}((2i+1)^2 -\frac{4}{9})\right)
    \frac{{(6\sqrt{3x})}^{2n+1}}{(2n+1)!}\\
    &=4\sqrt{3x}\sum_{n\geq 0}\frac{12^n}{(2n+1)!}
    \left(\prod_{i=0}^{n-1}(6i+1)(6i+5)\right)x^n\\
    &=8\sqrt{3x}\sum_{n\geq 0}S_nx^n,
    \end{split}
    \end{equation}
    which
    is nothing but \eqref{mathematica-find-1}.
    \item[(ii)]For   \eqref{mathematica-find-2},
    let
    \begin{equation}\label{ff}
    f(x)=\sum_{n=0}^{\infty}\frac{S_n}{2n+3}
    x^{2n+3}.
    \end{equation}
 Thanks to  (i), we get
    $$f'(x)=\sum_{n=0}^\infty S_nx^{2n+2}=\frac{x\sin(\frac{2}{3}\arcsin(6x\sqrt{3}))}{8\sqrt{3}}.$$
  From Theorem \ref{in-int}, it follows that
    \begin{align*}
    f(x)&=\frac{1}{8\sqrt{3}}\int_{0}^xt\sin(\frac{2}{3}\arcsin(6t\sqrt{3}))dt\\
    &=\frac{\frac{3}{4} \sin \left(\frac{4}{3} \sin ^{-1}\left(6 \sqrt{3} x\right)\right)-\frac{3}{8} \sin \left(\frac{8}{3} \sin ^{-1}\left(6 \sqrt{3} x\right)\right)}{3456 \sqrt{3}},
    \end{align*}
    which indicates
    \begin{equation}\label{value}
    f(1/6\sqrt{3})=\frac{1}{6144}.
    \end{equation}
    Therefore, we can deduce identity \eqref{mathematica-find-2} by invoking \eqref{ff} and \eqref{value} .

\end{itemize}
In addition, combining formulas \eqref{arcsin+arccos} and
 \eqref{cos(a+b)}, it is not difficult to obtain the following
result,
\begin{prop}\label{prop-1}
For all $t\in \mathbb{R}$ and $x\in [-1,1]$, we have
\begin{equation*}
\sin(\frac{\pi t}{2})\sin(t \arcsin(x))+\cos(\frac{\pi t}{2})\cos(t \arcsin(x))=\cos(t \arccos(x)).
\end{equation*}
\end{prop}
With above results, we are ready to prove Theorem \ref{theorem}.

\section{Proof of the Theorem \ref{theorem}}
\label{sec:Prove Conjecture}
We shall give two lemmas before giving the proof of Theorem \ref{theorem}.
The first lemma is Fermat's simple theorem \cite{fermat}.
\begin{lemm}\label{fermat-theorem}
If $a$ is any integer prime to $m$, and if $m$ is prime, then
$a^{m-1}\equiv1\,\pmod m.$
\end{lemm}

\begin{lemm}\label{lem-modp}
For any prime $p$, we have
\begin{equation*}
\frac{1}{p}{\binom{3p-2}{p-1}}\equiv -2\, \pmod p.
\end{equation*}
\end{lemm}
\begin{proof}
By applying
\begin{align*}
\frac{1}{p}{\binom{3p-2}{p-1}}
&=\frac{(3p-2)!}{p!(2p-1)!}\\
&=2\prod_{j=1}^{p-2}
\left(\frac{3p}{j+1}-1\right)
\end{align*}
and
\begin{align*}
\prod_{j=1}^{p-2}
\left(\frac{3p}{j+1}-1\right)&\equiv (-1)^{p-2}\\
&=-1\,\pmod p,
\end{align*}
the conclusion can be derived at once.

Now we can prove   Theorem \ref{theorem}.

\emph{Proof of Theorem \ref{theorem}.}  Firstly, we need to point out that \eqref{eq-conjecture-1} can be rewritten as follows:
\begin{equation*}
\cos(\frac{2}{3}\arccos(6x\sqrt{3}))=\sin(\frac{\pi}{3})\sin(\frac{2}{3}\arcsin(6x\sqrt{3}))
+\cos(\frac{\pi}{3})(1-24\sum_{k=1}^{\infty}T_kx^{2k}).
\end{equation*}
By Proposition \ref{prop-1}, for $k=1,2,3,\ldots$, if we can find $T_k$ such that  \begin{equation*}
1-24\sum_{k=1}^{\infty}T_kx^{2k}=\cos(\frac{2}{3}\arcsin(6x\sqrt{3})),
\end{equation*}
 then
we can prove  the first part of  Conjecture \ref{conjecture-1}.

By \eqref{stanley-series-2}, we see that
\begin{equation*}
\begin{split}
\cos(\frac{2}{3}\arcsin(6x\sqrt{3}))&=-\sum_{n\geq 0}\frac{4}{9^n}\left(\prod_{i=1}^{n-1}(6i+2)(6i-2)\right)
\frac{(6x\sqrt{3})^{2n}}{(2n)!}\\
&=-\sum_{n\geq 0}\frac{16^n}{3n-1}{\binom{3n}{n}}x^{2n}\\
&=1-\sum_{n\geq 1}\frac{16^n}{3n-1}{\binom{3n}{n}}x^{2n}.
\end{split}
\end{equation*}
For integers $n\geq 1$, if we set $T_n=\frac{16^n}{24(3n-1)}{\binom{3n}{n}}$,
 thus obtaining the desired sequence $(T_n)_{n\geq 1}$ for Conjecture \ref{conjecture-1}.

It is clear that
 $$\frac{16^n}{24(3n-1)}{\binom{3n}{n}}
 =16^{n-1}\left(2{\binom{3n-2}{n-1}}
 -{\binom{3n-2}{n}}\right)$$
 is an integer. One can refer to \cite{sun2} for details. In view of  Lemmas \ref{fermat-theorem} and \ref{lem-modp}, we can
get $T_p\equiv-2 \,\pmod p$ for any prime $p$. This completes the
proof of Theorem \ref{theorem}.

\end{proof}

%\section{The Parity of $S_n$}
%\label{sec: The Parity S_n}
%By \eqref{sn}, it follows that
%\begin{align*}
%S_n&=\frac{12^n}{2(2n+1)!}\left(\prod_{i=0}^{n-1}(6i+1)(6i+5)\right)\\
%&=\frac{2^{n-1}}{n!}\frac{3^{n}}{(2n+1)!!}\left(\prod_{i=0}^{n-1}(6i+1)(6i+5)\right),
%\end{align*}
%where $(2n+1)!!=1\cdot3\cdot5\cdots (2n-1)(2n+1).$
%
%For a prime $p$, recall that the $p$-adic evaluation of an integer $n$ is given by $$\nu_p(n)=\sup\{a\in \mathbb{N}:p^a|n\}$$
%and for any nonnegative integer $n$,  we hold
%$$\nu_p(n!)=\sum_{i=1}^{\infty}\left\lfloor\frac{n}{p^i}\right\rfloor.$$
%Let $n=2^m$ for any positive integer $m$, then
%\begin{align*}
%\nu_2(2^m!)&=\sum_{i=1}^{\infty}\left\lfloor\frac{2^m}{2^i}\right\rfloor\\
%&=2^m-1.
%\end{align*}
%If we suppose $n$ is not a power of two, then  there   exists some
%positive integer $l$ such that $2^l<n<2^{l+1},$
%\begin{align*}
%\nu_2(n!)&=\sum_{i=1}^{\infty}\left\lfloor\frac{n}{2^i}\right\rfloor\\
%&<\sum_{i=1}^{l}\frac{n}{2^i}\\
%&=n(1-\frac{1}{2^l}) <n-1.
%\end{align*}
%It follows  from the above results  that $S_n$ is odd if and only if
%$n$ is a power of two. This result partially confirms another
%conjecture of Z. W. Sun  \cite{sun}.

%%%%%%%%%%%%%%%%%%%%%%%%%%%%%%%%%%%%%%%%%%%%%%%%%%%%%%%%%%%%%%%%%%%%%%%%%%%%%%%%%%%
\noindent{\bf Acknowledgements.} I would like to thank the referee for valuable comments and suggestions.
This work was supported by NSFC (No. 11171283).

\end{CJK*}
\end{document}